\documentclass[a4paper,twoside,10pt,reqno]{article}
\usepackage[english]{babel}
\usepackage{graphicx}
\usepackage{epstopdf}
\usepackage[T1]{fontenc}
\usepackage[latin1]{inputenc}
\usepackage{amssymb,amsmath,amsthm,dsfont,mathrsfs}
\usepackage{amsfonts,euscript}
\usepackage{color}
\usepackage{authblk}
\usepackage{multirow}
\usepackage{mathtools}

\usepackage{authblk}
\usepackage{fullpage}
\usepackage{setspace}
\newtheorem{defi}{Definition}[section]
\newtheorem{teo}{Theorem}[section]
\newtheorem{pro}[teo]{Proposition}
\newtheorem{cor}[teo]{Corollary}

\newtheorem{ex}{Example}[section]

\newcommand{\nid}{\noindent }

\begin{document}

	\title{Tail dependence and smoothness}
	\author{Helena Ferreira}
	\affil{Universidade da Beira Interior, Centro de Matem\'{a}tica e Aplica\c{c}\~oes (CMA-UBI), Avenida Marqu\^es d'Avila e Bolama, 6200-001 Covilh\~a, Portugal\\ \texttt{helena.ferreira@ubi.pt}}

	\author{Marta Ferreira}
	\affil{Center of Mathematics of Minho University\\ Center for Computational and Stochastic Mathematics of University of Lisbon \\
		Center of Statistics and Applications of University of Lisbon, Portugal\\ \texttt{msferreira@math.uminho.pt} }
	
	\date{}
	
	\maketitle

\abstract{The risk of catastrophes is related to the possibility of occurring extreme values. Several statistical methodologies have been developed in order to evaluate the propensity of a process for the occurrence of high values and the permanence of these in time. The extremal index $\theta$ (Leadbetter \cite{Leadbetter1983} 1983) allows to infer the tendency for clustering of high values, but does not allow to evaluate the greater or less amount of oscillations in a cluster. The estimation of $\theta$ entails the validation of local dependence conditions regulating the distance between high levels oscillations of the process, which is difficult to implement in practice. In this work, we propose a smoothness coefficient to evaluate the degree of smoothness/oscillation in the trajectory of a process, with an intuitive reading and simple estimation. Application in some examples will be provided. We will see that, in a stationary process, it coincides with the tail dependence coefficient $\lambda$ (Sibuya \cite{Sibuya1960} 1960, Joe \cite{Joe1997} 1997), providing a new interpretation of the latter. This relationship will inspire a new estimator for $\lambda$ and its performance will be evaluated based on a simulation study.}

\nid\textbf{keywords:} {extreme values, smoothness coefficient, tail dependence coefficient}\\

\nid\textbf{AMS 2000 Subject Classification}: 60G70\\

\section{Introduction}
The occurrence of high values in a stochastic process can mean a natural, social or economic catastrophe, which has motivated the development of statistical models and techniques for extremes of random variables (see, e.g., Gomes and Guillou, \cite {Gomes2014} 2014 and their references). The unpredictability we would like to dominate is based on the propensity of the process for high values and the mean time permanency of these, usually measured by the arithmetic inverse of the extremal index $\theta$ (Leadbetter \cite{Leadbetter1983} 1983; Hsing \emph{et al.}~\cite{Hsing1988} 1988). Clustering of high values can be predicted in models that verify local dependency conditions D$^{(k)}(u)$, which regulate the distance between oscillations of the process relative to high levels $u$ (Chernick \emph{et al.}~\cite{Chernick1991} 1991). Under the validity of such conditions we can obtain expressions for the mean size $1/\theta$ of a cluster of high values. Not only the validation of local dependence conditions is difficult in practice, but also the estimation of $1/\theta$ does not give us information about the greater or less amount of oscillations in a cluster. In this work, we propose a measure to distinguish between processes with more oscillating trajectories from processes with smoother ones, which has an intuitive reading and is easy to estimate. The smoothness coefficient of a block of variables $\{X_i,\,n\leq i\leq m\}$ that we propose takes values in $[0,1]$ and grows with the degree of concordance of the variables. We will applied it in theoretical examples. We will also verify that, in a stationary process, it coincides with the tail dependence coefficient $\lambda$  (Sibuya \cite{Sibuya1960} 1960, Joe \cite{Joe1997} 1997), which gives us a new reading for this well-known coefficient in the  literature of extremes. The new representation for the tail dependence coefficient inspires an estimation procedure that will be analysed through a simulation study.

This paper is organized as follows: in Section \ref{scoef} we introduce the smoothness coefficient and present some properties and examples. In Section \ref{sinfer} we consider a new estimator for $\lambda$ and analyse its performance through simulation.

\section{The smoothness coefficient}\label{scoef}

Consider $\{X_i\}_{i\geq 1}$ a sequence of real random variables (r.v.'s) and denote $F_i$ the distribution function (d.f.) of $X_i$, $i\geq 1$. A natural way to evaluate the propensity for oscillations within a process $\{X_i\}_{i\geq 1}$ is to compare the expected number of oscillations in instant $i$,
\begin{eqnarray}\nonumber
\{F_i(X_i)\leq u<F_j(X_j)\},\, j=i-1,i+1,
\end{eqnarray} 
relative to real high levels $u$, with the expected number of exceedances of $u$,
\begin{eqnarray}\nonumber
\{F_j(X_j)>u\},\, j=i-1,i+1,
\end{eqnarray} 
around the instant $i$. Existing, at least, one exceedance between instants $n$ and $m$ ($n,m\in\mathbb{N}$), the expected total of oscillations will be closer of the expected total of exceedances, for $n\leq i\leq m$, in processes with more oscillating trajectories. We then propose as a summary measure of the result of this comparison between exceedances and oscillations, a coefficient with values in $[0,1]$, which increases with the concordance of the variables.

\begin{defi} The smoothness coefficient $S_{n,m}$ of $\{X_i\}_{n\leq i\leq m}$ is defined by
\begin{eqnarray}\label{smoothCoef}
S_{n,m}=1-\displaystyle\lim_{u\uparrow 1}\frac{E\left(\sum_{i=n}^{m}\sum_{j\in V(i)}\mathbf{1}_{\{F_i(X_i)\leq u<F_j(X_j)\}}|\sum_{i=n}^{m}\mathbf{1}_{\{F_i(X_i)>u\}}>0\right)}{E\left(\sum_{i=n}^{m}\sum_{j\in V(i)}\mathbf{1}_{\{F_j(X_j)>u\}}|\sum_{i=n}^{m}\mathbf{1}_{\{F_i(X_i)>u\}}>0\right)},
\end{eqnarray}
where $V(i)=\{i-1,i+1\}$, provided the limit exists.
\end{defi}

The proposed smoothness coefficient can naturally be expressed as a function of tail dependence coefficients 
\begin{eqnarray}\nonumber
\lambda(j|i)=\displaystyle\lim_{u\uparrow 1}P(F_j(X_j)>u|F_i(X_i)>u)\,.
\end{eqnarray} 
These summarize the behavior of the bivariate tails of a sequence and have been extensively studied and applied in the literature of extremes (see, e.g., Schmidt and Stadtm\"{u}ller \cite{Schmidt2006} 2006, Li \cite{Li2009} 2009, Ferreira and Ferreira \cite{Ferreira2014} 2014, and references therein).

\begin{pro}
The smoothness coefficient $S_{n,m}$ of $\{X_i\}_{n\leq i\leq m}$ satisfies
\begin{eqnarray}\nonumber
S_{n,m}=\frac{1}{2(m-n+1)}\displaystyle\sum_{i=n}^{m}\sum_{j\in V(i)}\lambda(j|i)\,,
\end{eqnarray}
provided $\lambda(j|i)$ exists for all $n\leq i\leq m$ and $j\in V(i)$.
\end{pro} 

\begin{proof}
	Observe that
	\begin{eqnarray}\label{eqFerreira1}
	\begin{array}{rl}
	S_{n,m}
	=&1-\displaystyle\lim_{u\uparrow 1}\frac{\sum_{i=n}^{m}\sum_{j\in V(i)}\left(P(F_j(X_j)>u)-P(F_i(X_i)> u,F_j(X_j)>u)\right)}{\sum_{i=n}^{m}\sum_{j\in V(i)}P(F_j(X_j)>u)}\\\\
	=& \displaystyle\lim_{u\uparrow 1}\frac{\sum_{i=n}^{m}\sum_{j\in V(i)}P(F_i(X_i)> u,F_j(X_j)>u)}{2(m-n+1)(1-u)}\\\\
	=& \displaystyle\frac{\sum_{i=n}^{m}\sum_{j\in V(i)}\lim_{u\uparrow 1}P(F_j(X_j)> u|F_i(X_i)>u)}{2(m-n+1)}\,.
	\end{array}
	\end{eqnarray}
\end{proof} 

This result points to the reading of $\lambda=\lambda(j|i)$, $j\in V(i)$, in a stationary process, as the smoothness coefficient for any block of variables $\{X_i,\,n\leq i\leq m\}$.

\begin{cor}\label{corLambda}
	If $\{X_i\}_{i\geq 1}$ is a stationary sequence with tail dependence coefficient $\lambda$, then
	\begin{eqnarray}\nonumber
	\lambda=S_{n,m},\, \forall\, 1\leq n< m\,.
	\end{eqnarray}
\end{cor}

Tail dependence increases with the concordance of the variables (Li, \cite{Li2009} 2009). We can therefore deduce the following properties from (\ref{eqFerreira1}).

\begin{pro}
	Let process $\{X_i\}_{n\leq i\leq m}$ have smoothness coefficient $S_{n,m}$. Then
	\begin{itemize}
		\item[(i)] $S_{n,m}\in [0,1]$;
		\item[(ii)] If $\{X_i,\,n\leq i\leq m\}$ are more concordant than $\{Y_i,\,n\leq i\leq m\}$, then $S_{n,m}^{(X)}\geq S_{n,m}^{(Y)}$.
	\end{itemize}
\end{pro}
\begin{proof}
Assertion (i) results from the coefficient  definition  and for (ii), observe that, if $\{X_i,\,n\leq i\leq m\}$ are more concordant than $\{Y_i,\,n\leq i\leq m\}$, then $P\left(\bigcap_{i=n}^{m}F_i(X_i)>u_i\right)\geq P\left(\bigcap_{i=n}^{m}F_i(Y_i)>u_i\right)$, $\forall u_i\in [0,1]$.
\end{proof}

In the bounds of the concordance relation, we have the independent and totally dependent variables. If all random pairs $\{(X_i,X_j)\}$, $j\in V(i)$, $n\leq i\leq m$, are independent we have $S_{n,m}=0$, whereas if they are totally dependent then $S_{n,m}=1$. 

In the context of max-stable processes, the independence or total bivariate dependence of the variables in $\{X_i,\,n\leq i\leq m\}$ is equivalent to the independence or total dependence of all variables. Thus, if $\{X_i\}_{i\geq 1}$ is max-stable, then $\forall 1\leq n\leq m$, we will have $S_{n,m}=0$ if and only if $\{X_i,\,n\leq i\leq m\}$ are independent and $S_{n,m}=1$ if and only  $\{X_i,\,n\leq i\leq m\}$ are dependent. For the context of max-stability, we also have the possibility of relating $S_{n,m}$ with the extreme coefficients  $\epsilon$ (Tiago de Oliveira \cite{Tiago1962} 1962/1963, Smith \cite{Smith1990}  1990), which allows the estimation of the coefficients $\lambda(j|i)$, by estimating expected values (Ferreira, \cite{Ferreira2013} 2013).


\begin{ex}
	Consider the $r$-factor model (Einmahl \emph{et al.}~\cite{Einmahl2012}, 2012)
\begin{eqnarray}\nonumber
X_n=\max_{s=1,...,r} a_{s,n}^{\alpha}Z_s^{\alpha},\,n\geq 1,
\end{eqnarray}
where factors $Z_s$, $s=1,...,r$, are independent and Fr\'echet($\alpha$) distributed r.v.'s, $\alpha>0$, and $\{a_{s,n},\,s=1,...,r\}_{n\geq 1}$ are non-negative constants such that $\sum_{s=1}^{r}a_{s,n}>0$. Variables in $\{X_n\}_{n\geq 1}$ are not identically distributed since each one of the $r$ factors $Z_1,...,Z_r$ contribute to the value of $X_n$ with weights $a_{s,n}$ updated over time $n$. Specifically we have
\begin{eqnarray}\nonumber
F_{n}(x)=\exp\left(-x^{-1}\sum_{s=1}^{r}a_{s,n}^{\alpha}\right),\, x>0,\,n\geq 1\,.
\end{eqnarray}
We have
\begin{eqnarray}\label{ex1Lambda}
\begin{array}{rl}
\lambda(j|i)=&\displaystyle\lim_{u\uparrow 1}\frac{P(F_i(X_i)>u,F_j(X_j)>u)}{P(F_i(X_i)>u)}\\\\
=&2-\displaystyle\lim_{u\uparrow 1}\frac{1-P(F_i(X_i)\leq u,F_j(X_j)\leq u)}{1-u}\,.
\end{array}
\end{eqnarray}
Observe that
\begin{eqnarray}\label{ex1Lambda}
\begin{array}{rl}
P(F_i(X_i)\leq u,F_j(X_j)\leq u)
=&\displaystyle P\left(X_i\leq 1/(-\ln u \left(\sum_{s=1}^{r}a_{s,i}^{\alpha}\right)^{-1}),X_j\leq 1/(-\ln u \left(\sum_{s=1}^{r}a_{s,j}^{\alpha}\right)^{-1})\right)\\\\
=&\displaystyle P\left(\bigcap_{s=1}^{r}Z_s^{\alpha}\leq \min\left(a_{s,i}^{-\alpha}/(-\ln u (\sum_{s=1}^{r}a_{s,i}^{\alpha})^{-1}),a_{s,j}^{-\alpha}/(-\ln u (\sum_{s=1}^{r}a_{s,j}^{\alpha})^{-1})\right)\right)\,.
\end{array}
\end{eqnarray}
Thus, for the dependence on the tail of $X_i$ and $X_j$, we have
\begin{eqnarray}\label{ex1Lambda}
\begin{array}{rl}
\lambda(j|i)
=&2-\displaystyle\lim_{u\uparrow 1}\frac{1-u^{\sum_{s=1}^{r}\max\left(a_{s,i}^{\alpha}/\sum_{s=1}^{r}a_{s,i}^{\alpha},a_{s,j}^{\alpha}/\sum_{s=1}^{r}a_{s,j}^{\alpha}\right)}}{1-u}\\\\
=&\displaystyle 2-\sum_{s=1}^{r}\max\left(a_{s,i}^{\alpha}/\sum_{s=1}^{r}a_{s,i}^{\alpha},a_{s,j}^{\alpha}/\sum_{s=1}^{r}a_{s,j}^{\alpha}\right)\,.
\end{array}
\end{eqnarray}
Denoting $b_{s,n}=a_{s,n}^{\alpha}/\sum_{s=1}^{r}a_{s,i}^{\alpha}$, $n\geq 1$, $s=1,...,r$, we have
\begin{eqnarray}\nonumber
\begin{array}{rl}
S_{n,m}=&\displaystyle\frac{1}{2(m-n+1)}\sum_{i=n}^{m}\left(2-\sum_{s=1}^{r}\max(b_{s,i-1},b_{s,i})+2-\sum_{s=1}^{r}\max(b_{s,i},b_{s,i+1})\right)\\\\
=& 2-\displaystyle\frac{1}{2(m-n+1)}\sum_{i=n}^{m}\sum_{s=1}^{r}\left(\max(b_{s,i-1},b_{s,i})+\max(b_{s,i},b_{s,i+1})\right)\,.
\end{array}
\end{eqnarray}
In the particular case of $a_{s,n}=a_s$, $\forall n\geq 1$, we have a constant sequence and $b_{s,n}=b_s$, $\forall n\geq 1$. Thus we obtain $\sum_{s=1}^{r}b_s=1$ and $S_{n,m}=1$. 
If $r=1$, then $\{X_n\}_{n\geq 1}$ is a sequence of totally dependent variables and we have $\lambda(j|i)=1$ and $S_{n,m}=1$. 
Under the special case of equally weighted factors, that is, $a_{s,n}=a_n$, $s=1,...,r$, $n\geq 1$, and $X_n=a_n^{\alpha}\displaystyle\mathop{\max}_{s=1,...,r}Z_s^{\alpha}$, we have 
$$
b_{s,n}=b_n\equiv \frac{a_n^{\alpha}}{r a_n^{\alpha}}=\frac{1}{r},\, s=1,...,r,\, n\geq 1 ,
$$
and therefore
\begin{eqnarray}\nonumber
\begin{array}{rl}
S_{n,m}=& 2-\displaystyle\frac{1}{2(m-n+1)}\sum_{i=n}^{m}r\left(\max(b_{i-1},b_{i})+\max(b_{i},b_{i+1})\right)\\\\
=&  2-\displaystyle\frac{1}{2(m-n+1)}\sum_{i=n}^{m}r\left(\frac{2}{r}\right)=1\,.
\end{array}
\end{eqnarray}
\end{ex} 

\begin{ex}
	(Temporary Failures Model or "Stopped Clock") Let $\{Y_n\}_{n\geq 1}$ be a sequence of independent and identically distributed (i.i.d.) variables and independent of the sequence of Bernoulli variables $\{Z_n\}_{n\geq 1}$. Consider notations $F(x)=P(Y_n\leq x)$, $n\geq 1$, and 
	$$
	p_{n,n+1,\dots,n+s}(i_0,i_1,\dots,i_s)=P(Z_n=i_0,Z_{n+1}=i_1,\dots, Z_{n+s}=i_s),
	$$
	$i_0,\dots,i_s\in \{0,1\},\,s\geq 1$. We denominate by temporary failures model, a sequence $\{X_n\}_{n\geq 1}$ defined as follows:
	\begin{eqnarray}\nonumber
	\begin{array}{l}
	X_1=Y_1\\\\
	X_n=\left\{
	\begin{array}{ll}
	X_{n-1}&,\,\textrm{ se }Z_n=0\\
	Y_n&,\,\textrm{ se }Z_n=1
	\end{array},\,n\geq 2.
	\right.
	\end{array}
	\end{eqnarray}
	Such designation relies on the interpretation of $\{Z_n\}_{n\geq 1}$ as a sequence of states corresponding to the registration or non-registration of values of $\{X_n\}_{n\geq 1}$. Thus, if, for example,
	$\{Z_1=1,Z_2=0,Z_3=0,Z_4=1,Z_5=1,Z_6=0,Z_7=1,Z_8=0,Z_9=0,Z_{10}=0,Z_{11}=1\}$, we will have, almost surely, $\{X_1=Y_1,X_2=Y_1,X_3=Y_1,X_4=Y_4,X_5=Y_5,X_6=Y_5,X_7=Y_7,X_8=Y_7,X_9=Y_7,X_{10}=Y_7,X_{11}=Y_{11}\}$. Zero sequences at the values of $\{Z_n\}_{n\geq 1}$ determine replicates of the last recorded value of $\{Y_n\}_{n\geq 1}$. If $n$ is the time, the zeros of $Z_n$ mean a stop of the register in time, keeping the last record. Let us consider a short-failures model to illustrate the smoothness coefficient calculation. In the short-failures model, we assume that $p_{n,n+1}=0$, i.e., it is almost impossible to lose two or more consecutive records of $\{Y_n\}_{n\geq 1}$.
	We start by deriving the common d.f.~of $X_n$:
	\begin{eqnarray}\nonumber
	\begin{array}{rl}
	P(X_n\leq x)=&P(X_{n-1}\leq x , Z_n=0)+P(Y_{n}\leq x , Z_n=1)\\
	=&F(x)p_{n-1,n}(1,0)+F(x)p_{n}(1)\\
	=&F(x)p_n(0)+F(x)p_n(1)\\
	=&F(x).
	\end{array}
	\end{eqnarray}
	Suppose, without loss of generality, that $F(x)=\exp(-1/x)$, $x>0$. For $u\in(0,1]$ and $v=(-\log u)^{-1}$, we have
	\begin{eqnarray}\nonumber
	\begin{array}{rl}
	&P(F(X_i)\leq u,F(X_{i+1}\leq u)\\
	=&P(X_{i}\leq v, X_{i+1}\leq v , Z_i=1,Z_{i+1}=1)\\
	&+P(X_{i}\leq v, X_{i+1}\leq v , Z_i=1,Z_{i+1}=0)\\
	&+P(X_{i}\leq v, X_{i+1}\leq v , Z_i=0,Z_{i+1}=1)\\
	=&F^2(v)p_{i,i+1}(1,1)+F(v)p_{i,i+1}(1,0)+F^2(v)p_{i-1,i,i+1}(1,0,1)\\
	=&u^2(p_{i,i+1}(1,1)+p_{i,i+1}(0,1))+up_{i,i+1}(1,0).
	\end{array}
	\end{eqnarray}
	Therefore,
	\begin{eqnarray}\nonumber
	\begin{array}{rl}
	\lambda(i+1|i)=&2-\displaystyle\lim_{u\uparrow 1}\frac{1-P(F(X_i)\leq u,F(X_{i+1}\leq u)}{1-u}\\
	=& 2-\displaystyle\lim_{u\uparrow 1}\frac{1-u^2(p_{i,i+1}(1,1)+p_{i,i+1}(0,1))-up_{i,i+1}(1,0)}{1-u}\\
	=&2-2(p_{i,i+1}(1,1)+p_{i,i+1}(0,1))-p_{i,i+1}(1,0)\\
	=&2-2(1-p_{i,i+1}(1,0))-p_{i,i+1}(1,0)\\
	=&p_{i,i+1}(1,0),
	\end{array}
	\end{eqnarray}
	and we obtain the smoothness coefficient given by
	$$
	\begin{array}{rl}
	S_{n,m}=&\displaystyle\frac{1}{2(m-n+1)}\sum_{i=n}^{m}(p_{i-1,i}(1,0)+p_{i,i+1}(1,0))\\
	=&\displaystyle\frac{1}{2(m-n+1)}\sum_{i=n}^{m}(p_{i}(0)+p_{i+1}(0)).
	\end{array}
	$$
	We can see that $S_{n,m}$ increases with the tendency to stop in the initial sequence records, as expected. With some more time-consuming calculations, we can extend the result to models with longer lasting failures. We note that in this short-failures model, the estimation of $p_{i,i+1}(1,0)$ allows us to estimate $S_{n,m}$. The estimation of $p_{i,i+1}(1,0)$ can be done from the natural estimation of $P(X_i=X_{i+1})=E(\mathbf{1}_{\{X_i=X_{i+1}\}})$, since, in general, $\{Y_n\}_{n\geq 1}$ and $\{Z_n\}_{n\geq 1}$ are unobservable sequences.
\end{ex}	

\section{A new estimator for $\lambda$ under stationarity}\label{sinfer}

The usual linear Pearson's correlation coefficient does not give us enough insight about the amount of dependence in the tails (Embrechts \emph{et al.}~\cite{Embrechts2002} 2002). Extreme values theory is the natural framework to address this topic. The tail dependence coefficient $\lambda$ is perhaps the most common measure of extremal dependency. Many other coefficients have been presented in the literature, most of them related to $\lambda$ (see, e.g., Schmidt and Stadtm\"uller \cite{Schmidt2006} 2006, Li \cite{Li2009} 2009, Ferreira and Ferreira \cite{Ferreira2018b} 2018, and references therein). The smoothness coefficient introduced here is another measure of tail dependence and from Corollary \ref{corLambda} it coincides with $\lambda$ under stationarity. Inference based on the definition in (\ref{smoothCoef}) is quite straightforward by taking the respective empirical counterparts. Thus, we can state a new estimator for $\lambda$ based on $S_{n,m}$, which we denote $\hat{\lambda}^{FF}$. More precisely, considering a stationary sequence $\{X_n\}_{n\geq 1}$ with marginal d.f.~$F$ and $U(u)$ and $E(u)$, respectively the number of upcrossings and the number of exceedances of a high level $u$ of $\{F(X_n)\}_{n\geq 1}$, we have
\begin{eqnarray}\label{estimLambda}
\hat{\lambda}^{FF}:=\hat{S}_{n,m}=1-\frac{U(u)}{E(u)}\,.
\end{eqnarray}

In the following we address a simulation study in order to analyse the performance of $\hat{\lambda}^{FF}$ in (\ref{estimLambda}). We also consider two estimators of $\lambda$ well-known and commonly used in literature, motivated by the second equality in (\ref{ex1Lambda}):
\begin{eqnarray}\nonumber
\begin{array}{ccc}
\hat{\lambda}^{LOG}:=2-\frac{\log C_n(u,u)}{\log u} & \textrm{ and }&
\hat{\lambda}^{SEC}:=2-\frac{1-C_n(u,u)}{1-u},
\end{array}
\end{eqnarray}
where
$$
C_n(u,u):=\frac{1}{n-1}\sum_{i=1}^{n-1}\mathbf{1}_{\{\hat{F}(X_i)\leq u,\hat{F}(X_{i+1})\leq u\}},
$$
and $\hat{F}$ corresponds to the empirical d.f.~of $F$. See Frahm \emph{et al.}~(\cite{Frahm2005} 2005) and references therein.
  
The simulations correspond to $200$ replicas of samples with size $n=1000$ from the following models for $\{X_n\}_{n\geq 1}$:
\begin{itemize}
	\item First-order max-autoregressive (Davis and Resnick \cite{Davis1989} 1989) denoted MAR(1):\\
	 $X_n=\max\left(cX_{n-1}, (1-c)Z_n\right)$, $0<c<1$, with $\{Z_n\}_{n\geq 1}$ a sequence of i.i.d.~r.v.'s with unit Fr\'echet d.f., as well as $X_0$ and thus $X_n$, $n\geq 1$. We have $\lambda=c$ (see, e.g., Ferreira and Ferreira \cite{Ferreira2012} 2012);
	\item First order moving-maximum (Davis and Resnick \cite{Davis1989} 1989) denoted MMA(1):\\
	 $X_n=\max\left(cZ_{n}, (1-c)Z_{n-1}\right)$, $0<c<1$, with $\{Z_n\}_{n\geq 1}$ a sequence of i.i.d.~r.v.'s with unit Fr\'echet d.f., as well as $Z_0$ and thus $X_n$, $n\geq 1$. We have $\lambda=\max(c,1-c)$ (see, e.g., Heffernan \emph{et al.}~\cite{Heffernan2007} 2007);
	\item First order autoregressive Yeh-Arnold-Robertson Pareto(III)  (Arnold \cite{Arnold2001} 2001), denoted YARP(1):\\
	 $X_n=\min\left(p^{-1/\alpha}X_{n-1}, \frac{1}{1-U_n}\varepsilon_n\right)$, where $\{\varepsilon_n\}_{n\geq 1}$, is a sequence of i.i.d.~r.v.'s coming from a Pareto(III)(0,$\sigma$,$\alpha$), i.e., $1-{F}_X(x)=\left[1+\left(x/\sigma\right)^{\alpha}\right]^{-1}$, $\sigma,\alpha>0$ and sequence $\{U_n\}_{n\geq 1}$ of i.i.d.~r.v.'s coming from Bernoulli$(p)$, $0<p<1$, independent of $\varepsilon_n$, $n\geq 1$. We consider $1/0\equiv +\infty$, $X_0\frown$Pareto(III)(0,$\sigma$,$\alpha$) and thus $X_n$, $n\geq 1$. We have $\lambda=p$ (Ferreira \cite{Ferreira2012b} 2012).
\end{itemize}

The absolute bias (abias) and the root mean squared error (rmse) derived from simulations are in Table \ref{tab1}, where we considered the high level $u$ given by the 95\% sample quantile. Quantiles 90\% and 99\% were also used but do not improve the results and are not reported. The values in bold correspond to the least absolute bias and the least root mean squared error obtained in each model. We can see that the three estimators have very similar performances. The estimator $\hat{\lambda}^{FF}$ proposed here, being of very simple application, thus constitutes a possible alternative.

\begin{table}[!h]\label{tab1}
	\caption{Simulation results corresponding to the absolute bias (abias) and root mean squared error (rmse) obtained for estimators $\hat{\lambda}^{FF}$, $\hat{\lambda}^{LOG}$ and $\hat{\lambda}^{SEC}$, considering $u$ the 95\% sample quantile, within models MAR(1) with parameter values $c=0.25,\,0.5,\,0.75$,  MMA(1) with parameter values $c=0.25,\,0.5,\,0.75$ and YARP(1) parameter values $p=0.25,\,0.5,\,0.75$.}
\begin{center}
{\def\arraystretch{2}\tabcolsep=10pt
	\begin{tabular}{lll|cc|cc|cc|}
		\cline{4-9}
&	&		&	\multicolumn{2}{|c|}{$\hat{\lambda}^{FF}$}		&	\multicolumn{2}{c|}{$\hat{\lambda}^{LOG}$}			&	\multicolumn{2}{c|}{$\hat{\lambda}^{SEC}$}		\\
	\cline{4-9}
& &		&	abias	&	rmse	&	abias	&	rmse	&	abias	&	rmse	\\
\hline
\multicolumn{1}{|c}{} &\multirow{3}{*}{MAR(1)}	&	\multicolumn{1}{|c|}{c=0.25}	&	\textbf{0.0559}	&	\textbf{0.0723}	&	0.0579	&	0.0745	&	0.0566	&	0.0724	\\
\multicolumn{1}{|c}{} & &	\multicolumn{1}{|c|}{c=0.50}	&	\textbf{0.0556}	&	0.0695	&	0.0557	&	\textbf{0.0680}	&	0.0561	&	0.0700	\\
\multicolumn{1}{|c}{} & &	\multicolumn{1}{|c|}{c=0.75}	&	0.0457	&	\textbf{0.0550}	&	0.0489	&	0.0594	&	\textbf{0.0456}	&	0.0551	\\
\hline
\multicolumn{1}{|c}{} & \multirow{3}{*}{MMA(1)}	&	\multicolumn{1}{|c|}{c=0.25}	&	\textbf{0.0163}	&	\textbf{0.022}	&	0.0198	&	0.0257	&	0.0277	&	0.0354	\\
\multicolumn{1}{|c}{} & &	\multicolumn{1}{|c|}{c=0.50}	&	0.0453	&	0.0581	&	\textbf{0.0430}	&	\textbf{0.0533}	&	0.0461	&	0.0587	\\
\multicolumn{1}{|c}{} & &	\multicolumn{1}{|c|}{c=0.75}	&	0.0439	&	0.0523	&	\textbf{0.0348}	&	\textbf{0.044} &	0.0440	&	0.0527	\\
\hline
\multicolumn{1}{|c}{} & \multirow{3}{*}{YARP(1)}	&	\multicolumn{1}{|c|}{p=0.25}	&	\textbf{0.0520}	&	\textbf{0.0678}	&	0.0531	&	0.0695	&	0.0524	&	0.0678	\\
\multicolumn{1}{|c}{} & &	\multicolumn{1}{|c|}{p=0.50}	&	0.0576	&	0.0695	&	\textbf{0.0503}	&	\textbf{0.0623}	&	0.0577	&	0.0699	\\
\multicolumn{1}{|c}{} & &	\multicolumn{1}{|c|}{p=0.75}	&	\textbf{0.0469}	&	\textbf{0.0604}	&	0.0485	&	0.0633	&	0.0471	&	0.0604	\\
\hline

	\end{tabular}
}
\end{center}
\end{table}

\end{document}